\newcommand{\bC}{{\mathbb{C}}}
\newcommand{\bP}{{\mathbb P}}

\newcommand{\bR}{{\mathbb{R}}}

\documentclass{amsart}

\usepackage{amsmath}
\usepackage{amssymb}
\usepackage{amsthm}

\theoremstyle{plain}
\newtheorem{thm}{Theorem}

\newtheorem{lem}{Lemma}
\newtheorem{cor}{Corollary}

\theoremstyle{definition}
\newtheorem{remark}{Remark}

\title[Hermitian Curvature and Plurisubharmonicity of Energy ]
{Hermitian Curvature and Plurisubharmonicity of Energy on Teichm\"uller Space}

\author{Domingo ~Toledo}
\address{ Mathematics Department, 
University of Utah, 
Salt Lake City, Utah 84112  USA  }
\email{ toledo@math.utah.edu }

 \thanks{Author partially supported by NSF grants 
DMS 0200877,  DMS-0600816 and by a grant from the Simons Foundation 208853}


\date{\today}

\bigskip

\begin{document}
\maketitle

\begin{abstract}
Let $M$ be a closed Riemann surface,  $N$ a Riemannian manifold of Hermitian non-positive curvature,  $f:M\to N$ a continuous map, and   $E$  the function on the Teichm\"uller space of $M$ that  assigns to a complex structure on $M$ the energy of the harmonic map homotopic to $f$.   We show that $E$ is a plurisubharmonic function on the Teichm\"uller space of $M$.  If $N$ has strictly negative Hermitian curvature, we characterize the directions in which the complex Hessian of $E$ vanishes.
\end{abstract}

\section{Introduction}

Let $M$ be a closed oriented surface,  let $N$ be a (compact) manifold of non-positive curvature, and fix a homotopy class of continuous maps from $M$ to $N$.  For each complex structure $J$ on $M$ there exists a harmonic map $f_J:(M,J)\to N$ in the given homotopy class.  In situations where this map is unique it varies smoothly with $J$ and its energy $E(f_J)$ defines a smooth function on the space of complex structures on $M$, which descends to a smooth function on the Teichm\"uller space of $M$.  See \cite{EL, SampsonENS, Tromba} for proofs of smooth dependence in several contexts.

Suppose now that the target manifold $N$ satisfies the stronger condition that it has non-positive Hermitian sectional curvature.  This is the curvature condition introduced by Siu  \cite{Siu} and Sampson \cite{Sampson} in order to apply harmonic mappings to rigidity theory.  By definition, it means that $R(X,Y,\bar X,\bar Y)\le 0$ for all $X,Y\in TN\otimes\bC$, where $R$ denotes the complex multilinear extension of the Riemann curvature tensor of $N$.   In particular, by choosing $X$ and $Y$ real, we see that the sectional curvatures $R(X,Y,X,Y)\le 0$, so  $N$ has non-positive sectional curvature.  See, for example,  \cite{ABCKT} for an exposition of how this condition is can be used in K\"ahler geometry. 

 The purpose of this paper is, first of all, to prove that  under the non-positive Hermitian curvature condition, this energy function is a plurisubharmonic function on Teichm\"uller space.   We state this formally as:
 
 \begin{thm}
 \label{thm-energypsh}
 Let $M$ be a compact Riemann surface, let $N$ be a compact manifold of non-positive Hermitian sectional curvature.  Fix a homotopy class of maps from $M$ to $N$, and assume  that for every complex structure $J$ on $M$  there is a unique harmonic map $f_J:(M,J)\to N$ in this class.   Let $E$ be the function on the Teichm\"uller space of $M$ that assigns to the equivalence class of the complex structure $J$ the energy $E(f_J)$.  Then $E$ is plurisubharmonic.
 \end{thm}
 
 We remark that the uniqueness assumption is typically satisfied.   For instance, if the sectional curvature of $N$ is strictly negative, then the harmonic map is unique unless its image is either a point  or a closed geodesic \cite{Hartman}.  If $N$ is a locally symmetric space of non-compact type, then $f$ is unique unless $f_*(\pi_1(M))$ is centralized by a semi-simple element in the group of isometries of the universal cover of $N$  \cite{Sunada}.  These locally symmetric spaces all satisfy $K_\bC \le 0$ \cite{Sampson} and are the natural setting for our applications.

  We also give sufficient conditions for this function to be strictly plurisubharmonic, see Theorem \ref{thm-strict}.   Briefly, if the Hermitian sectional curvature of the target is strictly negative and we have a family of harmonic maps with non-equal images, then the function is stricly plurisubharmonic, see Corollary \ref{cor-strict}.

	  But the main applications that we have presently in mind  involve situations  where the function is not strictly plurisubhamonic.  It is natural to characterize the tangent directions to Teichm\"uller space in which the complex Hessian of $E$ vanishes.    We partially do this  in the case that the Hermitian curvature of the target is strictly negative, see Theorem \ref{thm-equality},   which suggests a possible characteriztion that would be  more complete and  geometric, see Remark \ref{rem-conjectureschiffer}
 
 A significant special case of the plurisubharmonicity result is due to Tromba \cite{Tromba}.   He proves that this function is strictly plurisubharmonic in case that $N$ is also a negatively curved Riemann surface and the homotopy class consists of homotopy equivalences.    We follow Tromba's method, but a new ingredient is needed.  For a surface $N$  sectional (=  Gaussian) curvature and Hermitian sectional curvature are equivalent, but they differ in higher dimensions.  The new ingredient that we need is the  Micallef-Moore formula for the Laplacian of a complex variation.  It is interesting that this formula seems to be the only other appearance of Hermitian curvature in the literature, and its applications have been mainly to positively curved situations.

It is clear that  Theorem~\ref{thm-energypsh} also holds in more general situations.  We have chosen $N$ to be a compact manifold for simplicity in quoting existence theorems, but it will be clear from the proofs that it holds whenever we have existence and uniqueness theorems.    In particular it holds under suitable technical conditions for non-compact $N$, and, more generally, for harmonic maps $f:\tilde M\to \tilde N$ equivariant under suitable representations $\rho:\pi_1(M)\to G$, where $G$ is the group of isometries of $\tilde N$, the universal cover of $N$, and $\tilde M$ is the universal cover of $M$. 
 
 The motivation for this paper is an idea of Gromov that plurisubharmonicity of energy should give an alternative formulation of the rigidity theory of Siu and Sampson.  Recall that this theory shows that harmonic maps from a compact K\"ahler manifold to a quotient $N$ of a bounded symmetric  domain $D$ are pluriharmonic, and that pluriharmonic maps of sufficiently large rank  (where sufficiently large is a function of $D$), are holomorphic.  
 Gromov's idea is best explained when the domain $V$ is a smooth projective algebraic surface, decomposed as a \lq\lq curve of curves", that is, suppose we have an algebraic map $V\to Q$ to some parameter  algebraic curve $Q$ with fibres $V_q$,  which are  smooth and  connected for $q\in  Q\setminus\Sigma$ for some finite set $\Sigma\subset Q$ (the set of critical values).   Let $f:V\to N$ be a harmonic, hence pluriharmonic map.  Theorem \ref{thm-energypsh} shows that  $E(f_q)$ is a subharmonic function on $Q\setminus\Sigma$.  
 To illustrate how this can be put to work, we discuss a special case of  results in \S 4.6 of  \cite{Gromov}.   Namely,  if  the $V_q$ form a Lefschetz pencil and if one $f_q$ is holomorphic,  then $f$ is holomorphic.
  
 This paper grew out of conversations with Misha Gromov.   I am very grateful to him for suggesting that this result should be true, for many interesting conversations, and for his encouragement to write up this paper.    I am also grateful to Arnaud Beauville for some helpful suggestions, and to the referrees for suggestions that greatly improved the presentation of the paper/

\section{Plurisubharmonicity}

Let $M$ be a closed oriented surface and $J$ a compatible complex structure on $M$.  If $\alpha$ is a one-form on $M$, we write $J\alpha$ for the one form $\alpha\circ J^{-1} = -\alpha\circ J$.  More generally, if $A$ is an endomorphism of $TM$, we will write $A\alpha$ for $-\alpha\circ A$.  Then  $J\alpha = *\alpha$ where $*$ denotes the Hodge $*$-operator of any Riemannian metric in the conformal class of $J$.  In particular, $\int_M \alpha\wedge J\alpha = \int_M\alpha\wedge *\alpha$ is the square of the  $L^2$-norm of $\alpha$.  If $f:M\to N$ is a smooth map, let $<df\wedge Jdf>$ denote the two-form on $M$ obtained by combining the wedge product in $M$ with the Riemannian metric $<\ ,\ >$ on $f^*TN$.  Then 
\begin{equation}
\label{twovariableenergy}
\mathcal{E}(f,J) = \frac{1}{2} \int_M<df\wedge Jdf>
\end{equation}
is the energy of $f$, which is a function of $f$ and $J$.   In other words, if 
 $\mathcal{C}\subset C^\infty (M,End(TM))$ denotes the space of complex structures on $M$:
\begin{equation}
\label{spacecomplexstructures}
\mathcal{C} = \{J\in C^\infty(M,End(TM)): J^2 = -id\},
\end{equation}
and $C^\infty (M,N)$ denotes the space of $C^\infty$ maps from $M$ to $N$, then the energy is a function $\mathcal{E}:C^\infty(M,N)\times\mathcal{C}\to \bR$.

Now fix a compact manifold $N$ of non-positive Hermitian sectional curvature (in particular, of non-positive sectional curvataure), and fix homotopy class of maps from $M$ to $N$, that is, fix a connected component of $C^\infty(M,N)$.  For every complex structure $J\in\mathcal{C}$ there exists a harmonic map $M\to N$ in this component.  Under situations in which this map is unique it depends smoothly on $J$ and we can denote this harmonic map by $f_J$ and we let $E(J) = \mathcal{E}(f_J,J)$.  Then $E:\mathcal{C}\to \bR$ is a function that descends to Teichm\"uller space.  To prove Theorem \ref{thm-energypsh}, that $E$ is plurisubharmonic on Teichm\"uller space, is equivalent to proving that it is plurisubharmonic on $\mathcal{C}$ which means that its restriction to every germ of a complex curve in $\mathcal{C}$ is subharmonic.   This last statement can be checked in the following way.

Let $D$ be a small disk in $\bC$ centered at $0$, and let $J(s,t)$ be a family of complex structures on $M$ compatible with the orientation and depending holomorphically on the complex parameter $u = s+it\in D$.  Explicitly, this means that for each $u = s+it$, $J(s,t)\in \mathcal{C}$ satisfying the Cauchy-Riemann equations 
\begin{equation}
\label{holomorphicfamily}
J(s,t)^2 = - id \ \ \hbox{and} \ \ \frac{\partial J}{\partial t} = J \frac{\partial J}{\partial s}.
\end{equation}
Let $E(s,t) = \mathcal{E}(f(s,t),J(s,t))$, where we write $f(s,t)$ for  $f_{J(s,t)}$  Then $E$ is plurisubharmonic on $\mathcal{C}$ if and only of, for all such disks $D$ and families $J(s,t)$, we have $\Delta E(0,0) \ge 0$.  This is what we prove:

\begin{thm}
\label{thm-psh}
Let $E:D\to \mathbb{R}$ be the function just defined.  Then $\Delta E(0,0)\ge 0$.
\end{thm}

 An alternative and useful formulation of such a family $J(s,t)$ is  the following.  
 Give $M\times D$  the complex structure $\tilde J$ which is the direct sum of  $J(s,t)$ on $M\times (s,t)$ and the standard complex structure on $x\times D$ for each $x\in M$.  Explicitly,  $\tilde J (X + Y ) = JX + I Y$, where $X\in TM$, $Y\in TD$, $JX = J(s,t)X$ and $IY$ is the standard complex structure on $D$, that is, multiplication by $i$ on $TD$.   It is straightforward to check, since  $J$ a holomorphic function of $s+it$ (and since $J$ is integrable), that $\tilde J$ is an integrable almost complex structure on $M\times D$.
 
 The harmonic maps $f(s,t):(M,J(s,t))\to N$ can be assembled into a single map  $f:M\times D\to N$ which we think of as a variationf the  map $f|_{M\times\{0\}} : M\to N$.  We will perform computations on $M\times D$ using the following notation.   We will write simply $TM$ for the bundle $pr_M^*(TM)$ over $M\times D$, and $df$ for the section of $T^*M\otimes f^*TN$ which is the differential of $f$ restricted to directions tangent to $M$.  Then the energy function $E:D\to \bR$  under consideration  
\begin{equation}
\label{energy}
E(s,t) = \frac{1}{2} \int_M<df\wedge Jdf>.
\end{equation}
Recall that, for each $(s,t)\in D$, $f(\ \ ,(s,t))$ is harmonic with respect to $J(s,t)$:
\begin{equation}
\label{harmonic}
d_\nabla Jdf = 0
\end{equation}
where $d_\nabla :C^\infty ( T^*M\otimes f^*TN)\to C^\infty(\bigwedge^2 T^*M\otimes f^*TN)$ denotes the exterior derivative associated to the connection $\nabla$ on $f^*TN$ induced by $f$ from the Levi-Civita connection on $N$.

\begin{proof}[Proof of Theorem~\ref{thm-psh}]

We derive a formula for  $\frac{\partial^2 E}{\partial s^2}$ (similarly for $\frac{\partial^2 E}{\partial t^2}$) by using the following general principle: 

\begin{lem}
\label{lem-secondderivative}
 Suppose $g = g(x,y)$ is a function of two variables that is linear in $y$,  and  that  $(x(s),y(s))$ is a curve in the domain of $g$ so that for each $s$,  $x(s)$ is a critical point of the function $g(x,y(s))$, that is, $\frac{\partial g}{\partial x}(x(s),y(s)) = 0$.   Define a function $\phi$ by $\phi(s) = g(x(s),y(s))$.  Then 
 \begin{eqnarray}
  \label{secondone}
 \frac{\partial^2 \phi}{\partial s^2}  = \frac{\partial^2 g}{\partial x \partial y}\frac{ \partial x}{\partial s}\frac{ \partial y}{\partial s} + \frac{\partial g}{\partial y} \frac{\partial^2 y}{\partial s^2}\\
 \label{secondtwo}
  = -\frac{\partial^2 g}{\partial x^2}(\frac{\partial x}{\partial s})^2 + \frac{\partial g}{\partial y} \frac{\partial^2 y}{\partial s^2},
 \end{eqnarray}
 where the derivatives of $g$ are evaluated at $(x(s),y(s))$.
  \end{lem}
  \begin{proof}
  The first formula is clear, the second follows from the first by differentiating  $\frac{\partial g}{\partial x}(x(s),y(s)) = 0$.
 \end{proof}

Applying this lemma to our function $E(J) = \mathcal{E}(f_J,J)$,  Equation~\ref{secondone} gives
\begin{equation}
\label{secondderivativeone}
\frac{\partial^2 E}{\partial s^2} = \int_M ( <d_\nabla\frac{\partial f}{\partial s}  \wedge \frac{\partial J}{\partial s} df> + \frac{1}{2}  <df\wedge (\frac{\partial^2 J}{\partial s^2} df)>),
\end{equation}
(where the first term is obtained by differentiating (\ref{twovariableenergy}) separately in each variable: in direction $f$ keeping $J$ fixed (first variation of energy, before doing an integration by parts) and in direction $J$ keeping $f$ fixed), 
while Equation~(\ref{secondtwo}) gives 
\begin{equation}
\label{secondderivativetwo}
\frac{\partial^2 E}{\partial s^2} = -I(\frac{\partial f}{\partial s},\frac{\partial f}{\partial s} )+     \frac{1}{2} \int_M  <df\wedge (\frac{\partial^2 J}{\partial s^2} df)>,
\end{equation}
where $I(\frac{\partial f}{\partial s},\frac{\partial f}{\partial s} )$ is the {\em Index Form} or {\em Second Variation Form} of the energy $\mathcal{E}$ (with respect to its first variable, for fixed $J$), that is 

\begin{eqnarray}
\label{indexform}
I(\frac{\partial f}{\partial s},\frac{\partial f}{\partial s}) & = &  - \int_M<\frac{\partial f}{\partial s},d_\nabla(Jd_\nabla \frac{\partial f}{\partial s}) + \hat R(\frac{\partial f}{\partial s})> \\
{} & = &  \int_M(<d_\nabla \frac{\partial f}{\partial s}\wedge J d_\nabla \frac{\partial f}{\partial s}> - <\frac{\partial f}{\partial s},\hat R(\frac{\partial f}{\partial s})>) \nonumber.
\end{eqnarray} 

where $\hat R(\frac{\partial f}{\partial s})$ is the $f^*TN$-valued two form on $M$ defined by
\begin{equation}
\label{curvatureform}
\hat R(\frac{\partial f}{\partial s})(X\wedge Y) = R(\frac{\partial f}{\partial s},df(X))df(JY)  -  R(\frac{\partial f}{\partial s},df(Y))df(JX)
\end{equation}
for all $X,Y\in T_xM$. 
Observe that 
\begin{eqnarray}
\label{sectionalcurvature}
<\hat R(\frac{\partial f}{\partial s})(X\wedge JX),\frac{\partial f}{\partial s}> & = & R(\frac{\partial f}{\partial s},df (X), \frac{\partial f}{\partial s},df( X))  \\
{} & + &  R(\frac{\partial f}{\partial s},df( JX),\frac{\partial f}{\partial s}, df( JX)), \nonumber
\end{eqnarray}
where $R(X,Y,Z,W) = <R(X,Y)Z,W>$ is the curvature tensor of $N$.  In particular, it  has the same  sign as sectional curvature. The curvature assumption on $N$ implies  that $I$ is a positive (semi-) definite symmetric bilinear form.

Going back to formula (\ref{secondderivativeone}) and adding to it the same formula for the second derivative with respect to $t$, we  obtain the following formula for the Laplacian $\Delta = \frac{\partial^2}{\partial s^2} + \frac{\partial^2}{\partial t^2}$:
\begin{equation}
\label{laplacianone}
\Delta E = \int_M( <d_\nabla\frac{\partial f}{\partial s}  \wedge \frac{\partial J}{\partial s} df> +  <d_\nabla\frac{\partial f}{\partial t}  \wedge \frac{\partial J}{\partial t} df> + \frac{1}{2} <df\wedge( \Delta J )df>).
\end{equation}

Doing the same with formula~(\ref{secondderivativetwo}) we obtain
\begin{equation}
\label{laplacianrealindex}
\Delta E = - I(\frac{\partial f}{\partial s},\frac{\partial f}{\partial s}) - I(\frac{\partial f}{\partial t},\frac{\partial f}{\partial t}) +  \frac{1}{2}\int_M  <df\wedge( \Delta J )df>.
\end{equation}
Writing
\begin{equation}
\label{definitionw}
W = \frac{\partial f}{\partial t} + i \frac{\partial f}{\partial s},
\end{equation}
(the reason for the choice will be clear later; note that  $W = 2 i \frac{\partial f}{\partial u}$, where $u = s+it$), 
 the first two terms of  formula~(\ref{laplacianrealindex})  are the same as $-I(W,\overline W) = -I(\overline W,W)$ where, for complex vector fields $U,V$, $I(U,V)$ denotes the complex bilinear extension of $I$.  Thus an equivalent form of (\ref{laplacianrealindex}) is 
\begin{equation}
\label{laplaciancomplexindex}
\Delta E = -I(W,\overline W) +  \frac{1}{2}\int_M  <df\wedge( \Delta J )df> = - a + b,
\end{equation}
where $-a$ and $b$ are the first and second terms, respectively.   

Note that, differentiating  the second equation (\ref{holomorphicfamily}) with respec to $s$ and to $t$ and combining some terms we get  
\begin{equation}
\label{laplacianofJ}
\Delta J = J\  (\ \frac{\partial J}{\partial s}\ )^2.  
\end{equation}
It is easy to see that  $( \frac{\partial J}{\partial s})^2$ is a non-negative multiple of the identity, thus $b\ge 0$.
But  the non-positive curvature assumption forces, as remarked above, $a\ge 0$.  So proving the positivity of $b-a$ will require some work.  We will need both expressions (\ref{laplacianone}) and (\ref{laplaciancomplexindex}):  a suitable combination of an upper bound for the first two terms of (\ref{laplacianone}) and a lower bound for $I(W,\overline W)$ in (\ref{laplaciancomplexindex})  will give the desired inequality.

To simplify the notation define $J_0$ and $H$ to be:
\begin{equation}
\label{originofdisk}
J_0 = J(0,0),\ \ \ H = \frac{\partial J}{\partial s}(0,0). 
\end{equation}
Then the equations (\ref{holomorphicfamily}) and (\ref{laplacianofJ}) give

\begin{equation}
\label{derivatives}
 \ \ \frac{\partial J}{\partial t}(0,0) = J_0 H, \ \ \ (\Delta J )(0,0) = J_0H^2,
 \end{equation}
therefore the formulas  (\ref{laplacianone}) and (\ref{laplaciancomplexindex}) become:
\begin{equation}
\label{laplaciantwo}
\Delta E(0,0) = \int_M(<d_\nabla\frac{\partial f}{\partial s} \wedge H df> + <d_\nabla\frac{\partial f}{\partial t} \wedge J_0  H df> +  <df\wedge J_0H^2 df>),
\end{equation}
and 
\begin{equation}
\label{laplaciancomplexindextwo}
\Delta E (0,0) = - I(W,\overline W) + \int_M <df\wedge J_0H^2 df> = -a + b,
\end{equation}
where, as before, $-a$ and $b$ are the first and second terms, and $-a$ is also the sum of the first two terms of (\ref{laplaciantwo}).    Explicitly,
\begin{eqnarray}
\label{twoterms}
a &  =  & - \ \int_M(<d_\nabla\frac{\partial f}{\partial s} \wedge H df>  + <d_\nabla\frac{\partial f}{\partial t} \wedge J_0  H df>) \ \  = \ \  I( W,\overline W) \nonumber \\  
 b&= &\int_M <df \wedge J_0 H^2 df>,
\end{eqnarray} 
  where the first expression for  $a$ is minus  the sum of the first two terms of (\ref{laplaciantwo}), and the second expression for $a$  is from (\ref{laplaciancomplexindex}).
 
To prove that $\Delta E(0,0)\ge 0$ is the same as proving $a\le b$ which is the same as  
\begin{equation}
\label{inequality}
|\int_M (<d_\nabla\frac{\partial f}{\partial s} \wedge H df> + <d_\nabla\frac{\partial f}{\partial t} \wedge J_0  H df> )| \le \int_M <df\wedge J_0H^2 df>.
\end{equation}
The most straightforward estimate, obtained by applying Schwarz's inequality to each term on the left hand side, is too weak, because it treats the two terms on the left as separate entities.   To derive an efficient estimate we need to recognize this sum as a single inner product of two complex tensors.   

Decompose, as usual, the complexified tangent space $TM^\bC = T_{1,0}M\oplus T_{0,1}M$ and the complexified cotangent space $T^*M^\bC = T^{1,0}M\oplus T^{0,1}M$ into $\pm i$ eigenspaces for $J_0$.  Write $d'f$ for the restriction of  the complexified differential $df:TM^\bC\to f^*TN^\bC$ to $T_{1,0}M$ and $d''f$ for its restriction to $T_{0,1}M$.  Thus $d'f$ is a section of $T^{1,0}M\otimes f^*TN^\bC$ and $d''f$ is a section of $T^{0,1}M\otimes f^*TN^\bC$.  The complexified covariant derivative $d_\nabla$ splits as a sum $d_\nabla = d_\nabla' + d_\nabla''$.  In this notation, the harmonic equation for $f$ at $M,J_0$ is $d_\nabla'' d'f = 0$.

The complexification of the endomorphism $H$ of $TM$ anti-commuting with $J_0$ is of the form  $H = \mu + \bar\mu$, where $\mu$ is a section of $T^{0,1}M\otimes T_{1,0}M$ (and consequently $\bar\mu$ a section of $T^{1,0}M\otimes T_{0,1} M$).  Thus, in terms of a local complex coordinate $z = x+iy$ for $M$ with complex structure $J_0$.  
\begin{equation}
\label{beltrami}
H =m \frac{\partial}{\partial z}\otimes d\bar z + \bar m \frac{\partial}{\partial \bar z} \otimes dz,
\end{equation}
where $m$ is a smooth complex-valued function.   
Consequently we also have $J_0 H = i\mu - i \bar\mu$.  In this notation, the integrand in the first term of (\ref{inequality}) is
\begin{displaymath}
 <(d_\nabla' + d_\nabla'')\frac{\partial f}{\partial s}\wedge (\mu d'f + \bar\mu d''f)> ,
  \end{displaymath}
where, in accordance with our earlier notation, $\mu d'f = -d'f\circ \mu$ and $\bar\mu d''f = - d''f\circ \bar\mu$.  Note that $\mu d'f$ is a form of type $(0,1)$, given by $- m\frac{\partial f}{\partial z} d\bar z$ in local coordinates, while $\bar\mu d''f$ is of type $(1,0)$, and locally given by $ - \bar m \frac{\partial f}{\partial \bar z} dz$.  The bracket $<\ ,\ >$ denotes the complex bilinear extension of the inner product.   Therefore the Hermitian products will involve complex conjugation.

Expanding the above expression, we get the expression 
\begin{displaymath}
<d_\nabla\frac{\partial f}{\partial s}\wedge H df> =  <d_\nabla'\frac{\partial f}{\partial s}\wedge \mu d'f> + <d_\nabla''\frac{\partial f}{\partial s}\wedge\bar\mu d''f>.
\end{displaymath}
Similarly, the second term is
\begin{displaymath}
<d_\nabla\frac{\partial f}{\partial t}\wedge J_0 H df> =  <d_\nabla'\frac{\partial f}{\partial t}\wedge i \mu d'f> + <d_\nabla''\frac{\partial f}{\partial t}\wedge (-i)\bar\mu d''f>,
\end{displaymath}
thus their sum is
\begin{displaymath}
<d_\nabla '( \frac{\partial f}{\partial s} + i \frac{\partial f}{\partial t}) \wedge \mu d'f> + <d_\nabla'' (\frac{\partial f}{\partial s} - i \frac{\partial f}{\partial t}) \wedge\bar\mu d''f>
\end{displaymath}

which can be rewritten as 
\begin{displaymath}
<d_\nabla ' W \wedge \mu d'f> + <d_\nabla'' \bar W \wedge \bar\mu d''f> = 2 Re <d_\nabla' W\wedge J_0 \mu d'f>,
\end{displaymath}
where $W = \frac{\partial f}{\partial s} + i \frac{\partial f}{\partial t}$ as in (\ref{definitionw}).     We summarize:
\begin{equation}
\label{complexequation}
<d_\nabla\frac{\partial f}{\partial s}\wedge  H df>  + <d_\nabla\frac{\partial f}{\partial t}\wedge J_0 H df>  =  2 Re <d_\nabla' W\wedge J_0 \mu d'f>.
\end{equation}

Note that the hermitian form $\int_M <\alpha\wedge J_0 \bar \beta>$ is positive definite, and therefore so is the real form $\int_M  Re <\alpha\wedge J_0 \bar \beta>$.
Applying the Schwarz inequality to this bilinear form,  and then applying the inequality of arithmetic and geometric means gives the inequality
\begin{equation}
\label{schwarz}
|2 \int_M Re <d_\nabla' \overline W\wedge J_0\mu d'f>|\le \int_M <d_\nabla' \overline W\wedge J_0 d_\nabla''  W> + \int_M <\mu d'f\wedge J_0 \bar\mu d''f>,
\end{equation}
with equality if and only if $d''_\nabla  W$ and $\mu d'f$ are linearly dependent over $\bR$ (for the Schwarz inequality) and they have the same length (for equality of  the arithmetic and geometric means), thus if and only if 
\begin{equation}
\label{parallelvectors}
d_\nabla''  W = \pm \mu d'f,
\end{equation} 
One easily checks, say using the local coordinate expression (\ref{beltrami}), that  the second term on the right-hand side of (\ref{schwarz}) is
\begin{equation}
\label{secondterm}
 \int_M <\mu d'f\wedge J_0 \bar\mu d''f> = \frac{1}{2} \int_M < df\wedge J_o H^2 df>,
 \end{equation}
 which is  one of the terms we want in our desired inequality (\ref{laplaciantwo}).   Thus we turn to the first term on the right hand side of (\ref{schwarz}).    First observe that 
\begin{displaymath}
\int_M <d_\nabla ' \overline W\wedge J_0 d_\nabla ''  W>  =  \int_M <d_\nabla ''  W \wedge J_0 d_\nabla ' \overline  W>  =  2 \int_M |\nabla_{\bar z}  W|^2 dx\wedge dy, 
\end{displaymath}
the factor of $2$ resulting from $i d z\wedge \bar dz = 2 dx\wedge dy$. 
Compare now with the basic formula (2.3) of Micallef and Moore \cite{MM}, (which, by the appearance of this factor of $2$,  requires changing their coefficient of  $4$ in the first term to a $2$).   In our notation it reads:
\begin{equation}
\label{mm}
I( W, \overline W) = 2 \int_M <d_\nabla ' \overline W\wedge J_0 d_\nabla ''  W>  -\  4  \int_M R(\frac{\partial f}{\partial z},  W, \overline {\frac{\partial  f}{\partial  z}}, \overline W) dx\wedge dy.
\end{equation}
Observe that the integrand in the second term is Hermitian sectional curvature.   Thus this curvature makes its appearance through the Micallef-Moore formula.

 We now have all the ingredients that we need.  In order to have reasonable formulas to display, let us use the following notation
 \begin{eqnarray*}
a &  =  & - \ \int_M(<d_\nabla\frac{\partial f}{\partial s} \wedge H df>  + <d_\nabla\frac{\partial f}{\partial t} \wedge J_0  H df>) \ \  = \ \  I( W,\overline W) \nonumber \\  
 \alpha & = & \int_M <d_\nabla ' \overline W\wedge J_0 d_\nabla ''  W>  \\
 b&= &\int_M <df \wedge J_0 H^2 df> \\
 \rho & = &   \int_M R(\frac{\partial f}{\partial z},  W, \overline{\frac{\partial f}{\partial  z}}, \overline W ) dx\wedge dy,  \end{eqnarray*}
where $a$ and $b$ are as defined before in (\ref{twoterms}).    Note that the first three of these  quantities are non-negative, while the fourth, being Hermitian sectional curvature, is, by assumption, non-positive.    

The inequality (\ref{inequality}) that we want to prove is $a \le b$.  The inequality (\ref{schwarz}) proves that $a \le \alpha + \frac{b}{2}$.  The Micallef-Moore formula (\ref{mm})  and the equality $a = I( W,\overline W)$  show that $\alpha = \frac{a }{2}+ 2 \rho$.  Thus 
 \begin{displaymath}
 a  \le \frac{1}{2} (a  + b) + 2\rho
 \end{displaymath}
 equivalently,
 \begin{equation}
 \label{finalinequality}
a \le b + 4\rho
 \end{equation}
 which implies the desired inequality $a \le b$  since $\rho \le 0$ by the assumption that $N$ has non-positive Hermitian curvature.
 \end{proof}
  
 \section{Strict Plurisubharmonicity}
 
 Now we look at situations when the inequality (\ref{inequality}) must be strict.  One such situation is as follows.
 \begin{thm}
 \label{thm-strict}
 Suppose that  $R(X,Y,\bar X,\bar Y)<0$ whenever $X\wedge Y\ne 0$, in other words, $N$ has strictly negative Hermitian curvature.  Suppose that  $df$ is never zero, and suppose that the complex structure varies to first order, in other words, the class of $\mu$ is non-zero in $H^1(M,T_{1,0}M)$.   Equivalently, suppose $\mu$ represents a non-zero tangent vector to Teichm\"uller space.  Then $\Delta E (0) > 0$ for the variation (\ref{originofdisk}).
 \end{thm}
 
 Note that the following Corollary includes the result of Tromba mentioned in the introduction:
 
 \begin{cor}
 \label{cor-nobranch}
 Suppose that $N$ is a  Riemann surface with a metric of strictly negative curvature, and that the harmonic map $f:M\to N$  has at most generic singularities, in other words, $f$ is either non-singular or has only folds and cusps, and suppose $\mu$ is as above.  Then $\Delta E (0) > 0$ for the variation (\ref{originofdisk}).

 \end{cor}
 
 \begin{proof}
 To prove the theorem, suppose that equality holds in  (\ref{inequality}).  Then we must have that equality holds in (\ref{schwarz}) and in (\ref{finalinequality}).   If equality holds in (\ref{schwarz}) we must have that the condition (\ref{parallelvectors}) holds, namely 
\begin{equation}
\label{repeatparallelvectors}
d_\nabla''  W = \pm \mu d'f,
\end{equation} 
and if equality holds in (\ref{finalinequality}), we must have $\rho = 0$, which means that $\frac{\partial f}{\partial z}\wedge  W = 0$.   Since, by assumption, $d'f$ never vanishes, this implies that 
\begin{equation}
\label{proportional}
 W = \lambda d'f
\end{equation}
 for some smooth section $\lambda$ of the bundle $T_{1,0}M$.  Then 
\begin{equation}
\label{trivialvariation}
d_\nabla'' W =\bar  \partial \lambda d'f + \lambda d_\nabla'' d'f = \bar\partial \lambda d'f
\end{equation}
because $f$ is harmonic: $d_\nabla'' d'f = 0$.  
Comparing this equation with (\ref{repeatparallelvectors}) we must have $\mu = \pm  \bar \partial \lambda$, contradicting that $\mu\ne 0$ in $H^1(M,T_{1,0}M)$. 

To prove the Corollary, just obseve that for a non-singular map or a map with only folds or cusps, the rank of $df$ is at least one at each point, so $df$ is never $0$ and the Theorem applies.
 \end{proof}
 
This argument generalizes easily to the situation when $df$ has zeros.  The vector field $\lambda$ then has singularities, but defines a current which localizes the deformation class $\mu$.

 Fix $J_0$ as before and a direction $H$ at $J_0$, or, what is the same by (\ref{beltrami}), choose $\mu\in A^{0,1}(M,T_{1,0}M)$, (where $A^{p,q}$ denotes smooth forms of type $(p,q)$).   The class $[\mu]\in H^1(M,T_{1,0}M)$ represents a tangent vector to Teichm\"uller space at $J_0$.  We want to give a necessary and sufficient condition for the complex hessian of $E$ at $J_0$  to vanish in the direction $[\mu]$: $\bar\partial\partial _{jJ_0} E (\mu) = 0$. 

If $f$ is not constant then $df$ has isolated zeros.  Namely, $d_xf=0$ if and only if $d_x'f =0$ an the latter is a holomorphic section of $T^{1,0}M\otimes f^*TN^\bC$.   If $d_xf = 0$, choose a local holomorphic coordinate $z$ centered at $x$ and a local holomorphic trivialization of $f^*TN^\bC$.  Then $d'f = \frac{\partial f}{\partial z} dz$ and 
\begin{equation}
\label{zero}
\frac{\partial f}{\partial z} = z^m v(z)
\end{equation}
where $v(z)$ is a holomorphic vector function of $z$ with $v(0)\ne 0$ and $m\ge 1$.  Let $x_1,\dots ,x_k$ be the zeros of $d'f$, and let $m_i$ be the exponent $m$ for this zero (the multiplicity of $x_i$),  and  let 
\begin{equation}
\label{zerodivisor}
Z = \sum_i m_i \ x_i 
\end{equation}
be the zero divisor of $d'f$.    We write $|Z| = \{x_1,\dots,x_k\}$ for the (set-theoretic) support of $Z$, we write $\mathcal{O}(Z)$ for the line bundle over $(M,J_0)$ with canonical section $\sigma$ that defines the divisor $Z$.

\begin{thm}
\label{thm-equality}
Suppose that, as in Theorem~\ref{thm-strict}, $N$ has strictly negative Hermitian curvature, and suppose that at a point $J_0\in \mathcal{C}$ and tangent direction $\mu\in T_{J_0}\mathcal{C}$ with $[\mu]\ne 0\in H^1(M,T_{1,0}M)$, the complex Hessian of $E$ vanishes:  $\Delta E (0)  = 0$ for the variation (\ref{originofdisk}).  Then
\begin{enumerate}
\item The zero set $|Z|$ of $df$ is not empty.

\item The  closure of  $d'f(T_{1,0}M)$ in $f^*TN^\bC$ is a line sub-bundle  $L\subset f^*TN^\bC$,   $L\cong\mathcal{O}(Z)\otimes T_{1,0}M$.  

\item The variation field $ W $ is a $C^\infty$ section of $L$, and the equation (\ref{proportional}) holds for a smooth section $\lambda$ of $T_{1,0}(M\setminus|Z|)$.   In particular, $ W$ is tangent to $f(M)$.  In other words, the image of $M$ does not move to first order under the deformation.

\item   At each $x_i\in |Z|$ the section $\lambda$ of $T_{1,0}(M\setminus |Z|)$ defines a residue current $\alpha_i$.  The variation field $W$ does not vanish identically on $Z$ in the sense that not all of these currents $\alpha_i$ can vanish.

\item The section $\lambda$ on $M\setminus |Z|$ extends to  a current on $M$ and we have the equation of currents 
\begin{equation}
\label{distributionalequation}
 \bar\partial \lambda = \pm \mu - \sum \alpha_i.
\end{equation}

In particular, $[\mu]$ has a representative with support in $Z$.

\item The derivative of $E$ vanishes at $J_0$ in the direction $\mu$: $d_{J_0} E(\mu) = 0$.
\end{enumerate} 
\end{thm}

\begin{cor}
\label{cor-strict}
Suppose that $N$ has strictly negative Hermitian sectional curvature and that in the variation (\ref{originofdisk}) the image of $f$ varies to first order, in other words, $W$ is not tangent to $f(M)$.  Then $\Delta E(0,0) >0$.
\end{cor}

\begin{remark}
\label{rem-simplezeros}
The conclusions of the Theorem are easier to state in the case that all the multiplicities $m_i = 1$.   Then the fourth statement says that we cannot have $ W(x_i) = 0$ for all $i$, in other words,  some of the zeros of $df$ have to move under the deformation.  The fifth statement says that $\mu$ is cohomologous, as a current, to a sum of delta functions at the $x_i$ for  which $ W(x_i)\ne 0$, in other words, the variation of complex structure can be concentrated at the zeros of $df$ that move under the deformation. 
\end{remark}

\begin{remark}
\label{rem-schiffer}
The situation described in the theorem does occur, and there is a classical example:  branched covers and Schiffer variations.   Namely, let $N$ be itself a hyperbolic Riemann surface, and let $f:(M,J_0)\to N$ be a holomorphic branched cover.  Suppose, for simplicity, that $f$ is a double cover, so that all the zeros of $d'f$ have multiplicity one and  we are  in the situation of  Remark~\ref{rem-simplezeros}.  Let $p\in M$ be one of the branch points, and let $q = f(p)\in N$.  If we move $q$ and we take replace $(M,J_0)$ by the double branched cover of $N$ branched over $q$, we get a family $(M,J_q)$ where the variation $\mu$ is concentrated at the point $p$ that moves under the deformation.  This is called a (first order, global)  {\em Schiffer variation}, see, say,  Example 1 of Chapter 2, section 3 of \cite{MK}.   The harmonic map $f_q$ is still holomorphic, and having constant degree as $q$ varies, it has constant energy and thus $d E = \Delta E = 0$ in this family.   This is the motivating example behind the theorem.   Examples for zeros of arbitrary order are given by  higher order Schiffer variations.
\end{remark}

\begin{remark}
\label{rem-conjectureschiffer}
More genrally, suppose $N$ has strictly negative Hermitian sectional curvature but it is of any dimension, and suppose that $N_0$ is a hyperbolic Riemann surface.   Suppose that $f_0:(M,J_0)\to N_0$ is a branched cover as in Remark \ref{rem-schiffer}, and, for simplicity, assume it to be a double cover.  Let $(f_0 )_q:(M,J_q)\to N_0$ be as in Remark \ref{rem-schiffer}.  Suppose further that $g:N_0\to N$ is a harmonic map.  Define maps $f_q = g\circ (f_0)_q:(M,J_q)\to N$.  We get a family of harmonic maps of constant energy.   It is natural to ask if (in the case that $d'f$ has simple zeros), this is the only way to obtain directions in which the complex Hessian vanishes (and the corresponding construction for zeros of any order).  I this were true, it would give a new proof of some well-known theorems on factorization theorems for harmonic maps through holomorphic maps of Riemann surfaces.  See, for example,  the discussion of the fibrations and of factorizations in Chaptes 2, 4, 6 of   \cite{ABCKT} and the references given there.    In fact a localized version of some of these theorems in which a compact K\"ahler manifold is replaced by a neighborhood of a suitable complex curve in it.

\end{remark}

\begin{proof}[Proof of Theorem~\ref{thm-equality}.]
The first statement follows from  Theorem~\ref{thm-strict}: if  $df$ never vanishes  then $\Delta E(0)>0$.   For the second statement, the line sub-bundle $L\subset f^*TN^\bC$ is spanned, in $M\setminus |Z|$, by the vectors $\partial f/\partial z$, and in a neighborhood of each $x_i$, by the vectors $v(z)$ in (\ref{zero}).  It is clear that the image of $d'f$ coincides with $L$ on $M\setminus |Z|$, and from (\ref{zero}) that over the domain of the local coordinate $z$, the closure of the image of $d'f$ coincides with $L$.   It is also clear that $L\cong T_{1,0}M\otimes \mathcal{O}(Z)$, in fact, under this isomorphism, the bundle map $d'f:T_{1,0}\to T_{1,0}\otimes\mathcal{O}(Z)$ coincides, up to a constant factor, with multiplication by the canonical section  $\sigma$ of $\mathcal{O}(Z)$.

The third statement follows, as before, from the strict negativity of Hermitian curvature, that forces $\frac{\partial f}{\partial z}\wedge W = 0$, hence $W$ lies in $L$ since it does in the dense open set where $\frac{\partial f}{\partial z}\ne 0$.    This means that (\ref{proportional}) holds on $M\setminus |Z|$, for some smooth section of $\lambda$ of $T_{1,0}(M\setminus |Z|)$, but we have to examine it more closely on $|Z|$.   Fix a zero of $d'f$ and let $z$ be a local coordinate centered at this zero and 
 write $\lambda(z) = l(z)\frac{\partial}{\partial z}$  for some smooth function $l$ defined for $z\ne 0$.  Then, using (\ref{zero}),  (\ref{proportional}) becomes
\begin{equation}
\label{lambdaeta}
 W = \lambda d'f = l(z) z^m v(z) = \eta(z) v(z),
\end{equation}
where $\eta$ is smooth for all $z$.   Moreover, we get that (\ref{repeatparallelvectors}) becomes
\begin{equation}
\label{dbar}
\bar\partial l z^m v(z) = \pm m(z) z^m v(z) d\bar z.
\end{equation}
and therefore we get an equality of differential forms on $\{z\ne 0\}$:
\begin{equation}
\label{predistribution}
\bar\partial \lambda = \bar\partial( l\frac{\partial}{\partial z} )= \pm m \frac{\partial}{\partial z} \otimes d\bar z = \pm \mu\ \ \ \hbox{for}\ z\ne 0.
\end{equation}
but we need to check what the distributional derivative is at $0$.  To this end, fix a disk $D_\rho = \{|z|<\rho\}$ centered at $0$.   On $D_\rho$  (\ref{lambdaeta}) gives that $\lambda = l(z)\frac{\partial}{\partial z}$ where 
\begin{equation*}
l(z) = \frac{\eta(z)}{z^m}\ \hbox{where } \eta\hbox{ is smooth.}
\end{equation*}
This type of singularity of $\lambda$, even though it is not integrable for $m>1$,  still defines a {\em principal value current}, namely the linear functional on $A_c^{0,1}(D_\rho ,T^{1,0}\otimes T^{1,0})\subset A^{0,1}(M,T^{1,0}\otimes T^{1,0}) $ that assigns to $\phi(z) dz^2\bar dz$, where $\phi\in C_c^\infty(D_\rho)$, the number 

\begin{equation*}
\lambda[\phi dz^2\bar dz] = \lim_{r\to 0} \frac{1}{2\pi i} \int_{r<|z|<\rho} \ \frac{\eta(z)}{z^m} \phi(z) dz\bar dz.
\end{equation*}
  It is standard that this principal value limit exists.
Its  distributional derivative is   the linear functional $\bar \partial \lambda $ on the space of compactly supported quadratic differentials defined by 
\begin{equation*}
\label{distributionalderivative}
(\bar \partial \lambda) [\phi dz^2 ] = - \lim_{r\to 0} \frac{1}{2\pi i} \int _{r< |z|\le \rho}\frac{\eta(z)}{z^m} \bar\partial \phi dz
\end{equation*}
Integrating by parts we get 
\begin{equation*}
\label{stokes}
\bar\partial\lambda[\phi dz^2] = \lim_{r\to 0} \Big(\frac{1}{2\pi i} \int_{|z| = r} \frac{\eta(z)}{z^m} \phi(z) dz + \frac{1}{2\pi i} \int_{r\le |z|\le \rho} \bar\partial l(z)  \phi(z) dz\ \Big)
\end{equation*}

The second integral converges to $\frac{1}{2\pi i} \int_{|z| \le\rho} m(z) \phi(z) dzd\bar z$, while the first integral converges to what is called the {\em residue current}
\begin{equation}
\label{residuecurrent}
\alpha[\phi dz^2] = \lim_{r\to 0} \frac{1}{2\pi i} \int_{|z| = r} \lambda\cdot  \phi dz^2 \ = \ \lim_{r\to 0}\frac{1}{2\pi i}\int l(z) \phi(z) dz.
\end{equation}
It is standard and easy to see that this limit exists and has an explicit formula:
\begin{equation}
\label{deltas}
\alpha[\phi]  =  \sum_{j+l = m-1} c_j \frac{\partial^j\eta}{\partial z^j}(0)\frac{\partial^l \phi }{\partial z^l}(0),
\end{equation}
for suitable constants $c_j$, 
in other words, a sum of derivatives of delta functions at $0$ with coefficients multiples of the $z$-derivatives of $\eta$ at $0$.    Doing this at each zero $x_i$ of $df$, calling the resulting current (\ref{residuecurrent}, \ref{deltas}) $\alpha_i$, we get the equation of currents:
\begin{equation}
\label{distributionalequation}
 \bar\partial \lambda = \pm \mu - \sum \alpha_i
\end{equation}
These formulas prove the fourth and fifth statements.  Namely, (\ref{distributionalequation}) and (\ref{residuecurrent},\ref{deltas}) prove statement (5), and, if all residue currents of statement (4) vanished, then we would get all $\alpha_i = 0$, hence $\bar \partial\lambda = \mu$, contradicting the assumption $[\mu]\ne 0$.

Observe that formula (\ref{deltas}) shows that  statement (5) of the theorem is equivalent to saying: not all partial derivatives $\frac{\partial^j \eta_i}{\partial z^j}(0), \ 0\le j\le m_i-1$, $1\le i \le k$, vanish, where $\eta_i$ is as in (\ref{lambdaeta}) for the zero $x_i$ of $d'f$.

Finally, for the last statement,  recall the well-known formula

\begin{equation}
\label{}
d_{J_0}E(\mu) = c\  \Re\Big(  \int_M \mu \cdot Q \Big),
\end{equation}
(for some constant $c\ne 0$ that is not important here) 
which is obtained by differentiating (\ref{energy}) with respect to $J$, and where $Q$ is the holomorphic quadratic differential (Hopf differential) of $f$, namely $Q =( f^*h)^{(2,0)} = <d'f,d'f>$ where $h$ is the metric tensor of $N$ and $\mu\cdot Q$ denotes the natural pairing with values in $(1,1)$-forms.   See Theorem 3.1.3 of \cite{Tromba} for a proof and \cite{Wentworth} for a more general statement as well as a history of this fomrula.

Since $\bar\partial  Q = 0$,  (\ref{distributionalequation}) gives  $\bar\partial(\lambda\cdot Q) = \mu\cdot Q - (\sum \alpha_i)\cdot Q$, hence 
\begin{equation*}
\int_M \mu\cdot Q = \sum \alpha_i(Q).
\end{equation*}
 Since $Q = <d'f,d'f>$ vanishes at the zeros $x_i$ of $d'f$, to multiplicity at least $2m_i$, we get $\alpha_i(Q) = 0$ for all $i$, thus $d_{J_0} E (\mu) = 0$.

\end{proof}

\section{An application to rigidity theory}

We now sketch an application, due to Gromov \cite{Gromov}, of the plurisubharmonicity of energy to the Siu-Sampson rigidity theory.   Recall that the theory originated in Siu's method for showing that a harmonic map $f:V\to N$ is holomorphic, where $V$ is a compact K\"ahler manifold and $N$ is a Hermitian manifold with universal cover is a bounded symmetric domain $D$.  For any harmonic $f:V\to N$ as above, Siu proves  several theorems, including:

\begin{itemize}
\item The map $f$ is pluriharmonic: its restriction to any (germ of a) complex curve is harmonic.
\item If $D$ is irreducible and not the hyperbolic plane, and the rank of $f$  equals  the dimension of $D$, then $f$  is holomorphic or anti-holomorphic.
\item If $D = B^n$, the unit ball in $\bC^n$, $n\ge 2$, and the rank  of $f$ is $> 2$, then $f$ is holomorphic or anti-holomorphic.
\item For each irreducible $D$ other than the hyperbolic plane there is an integer $r(D)< dim_\bR(D)$ so that if the  rank of $f$ is  $>r(D)$, then $f$  is holomorphic or anti-holomorphic.

\end{itemize}

Unlike the case of $B^n$, where $r(B^n) = 2$, the number $r(D)$ can be rather large.  For example, let  $D = D_{p,q}$, $1\le p \le q$ be the generalized ball (the symmetric space of $SU(p,q)$,) of complex $p$ by $q$ matrices $Z$ with $Z^*Z<I_p$.  Then there is a totally geodesically embedded $D_{p-1,q-1}\times D_{1,1}\subset D_{p,q}$ of block-diagonal matrices, and it is easy to find a discrete group $\Gamma$ acting freely and co-compactly on $D_{p,q}$ with a subgroup of the form $\Gamma_1\times \Gamma_2$, each factor acting discretely and co-compactly on $D_{p-1,q-1}$, $D_{1,1}$ respectively.  Letting $V = V_1\times V_2$ be the quotient manifold with the corresponding product decomposition, we obtain a holomorphic map $f:V_1\times V_2\to W = \Gamma\backslash X$.  Now deform the complex structure on $V_2$ to $V_2'$ so that the harmonic map in the homotopy class $f|_{V_2}$ is not holomorphic for the complex structure of $V_2'$ to obtain a harmonic map $g:V_1\times V_2'$ of rank $2((p-1)(q-1)+1)$ that is not holomorphic.  Thus $r(D_{p,q})\ge 2((p-1)(q-1)+1)$ (and in fact equality holds, see \cite{SiuJDG}).

This example is extremely special, and we would like to have theorems to the effect that harmonic maps of low rank substantially different from this example are holomorphic or anti-holomorphic.  One way of excluding this example would be to require that $f_*(\pi_1(V))$ be Zariski dense in $G$, the group of biholomorphisms of $X$.   We illustrate how plurisubharmonicity of energy can be used by sketching a proof of the next theorem, which is a special case of Gromov's results in \S 4.6 of  \cite{Gromov}.   
   
\begin{thm}
\label{thm-amplecurve}
Let $V$ be a smooth projective algebraic surface, let $N$ be a Hermitian manifold with universal cover an irreducible bounded   symmetric domain $D$, let $f:V\to N$ be a harmonic map.   Let $\hat V\to \bP^1$ be a Lefschetz pencil, and suppose that there is a generic fiber $V_q$, $q\in \bP^1$, so that $f|_{V_q}$ is holomorphic.  Then $f$ is holomorphic.
\end{thm}

Recall that if $V\subset \bP^m$ is a smooth surface, a {\emph Lefschetz pencil} is the family of cuves obtained by fixing a linear $\bP^{m-2}\subset\bP^m$ transverse to $V$ and taking the intersection of $V$ with the pencil of $\bP^{m-1}$'s containing the fixed $\bP^{m-2}$.  This gives a family $V_q$ of curves parametrized by $a\in Q = \bP^1$.  The map $V\to Q$ is only a rational map, not defined at the basepoints of the pencil, namely the $d$ points $x_1,\dots ,x_d$  (where $d$ is the degree of $V$) where $\bP^{m-2}$ intersects $V$.   If $\hat V$ is the blow-up of $V$ at these points, then there is a well-defined holomorphic map $\hat V\to Q = \bP^1$ which fits our original setting.  This point will be mostly ignored.

\begin{cor}
\label{cor-zariskidense}
Let $V$, $\hat V\to \bP^1$, $N$ and $D$  be as above.  Let $g:V\to N$ be a smooth map and assume:
\begin{enumerate}
\item $g_*(\pi_1(V))$ is Zariski dense in the group $G$ of biholomorphisms of $D$.
\item For a generic fiber $V_q$ of the Lefschetz pencil, $g|_{V_q}$ is homotopic to a holomorphic map.
\end{enumerate}
Then $g$ is homotopic to a holomorphic map.
\end{cor}

\noindent{\sl Sketch of Proof of Theorem \ref{thm-amplecurve}}:  

\begin{enumerate}
\item By \cite{Siu} the map $f$ is pluriharmonic, thus for $q\in \bP^1\setminus\Sigma$, the map $f_q = f|_{V_q}$ is harmonic.   
  
\item  Strictly speaking, Theorem \ref{thm-energypsh} applies to the maps $f_{\tilde q}:V_{\tilde q}\to N$ for $\tilde q\in \overbrace{\bP^1\setminus\Sigma}$, the universal cover of $\bP^1\setminus\Sigma$ to Teichm\"uller space, giving that $E(f_{\tilde q})$ is plurisubharmonic on $ \overbrace{\bP^1\setminus\Sigma}$.    But the map $\tilde q\to V_{\tilde q}$ of $ \overbrace{\bP^1\setminus\Sigma}$ to Teichm\"uller space is  invariant under the subgroup of the mapping class group of $V_{\tilde q}$ that preserves the homotopy class of $f_{\tilde q}$, thus it descends to $\bP^1\setminus\Sigma$, thus Theorem \ref{thm-energypsh} applies.

\item  By Theorem \ref{thm-energypsh}, $E(q) = E(f_q)$ is a plurisubharmonic function on $\bP^1\setminus\Sigma$.
\item $E$ extends to a continuous function on $\bP^1$.  This requires some work, see \S 4.6  of \cite{Gromov}.
\item Suppose that for some $q_0\in Q\setminus \Sigma$ the map $ f_{q_0}$ is holomorphic, and denote this map simply by $f_0$.  A simple argument based on Wirtinger's inequality gives that  all $f_q$, $q\in Q$, are holomorphic.  Namely, for each $q\in Q\setminus\Sigma$ we have the inequalities
\begin{equation*}
\int_{V_q}f_q^*\omega\le A(f_q)\le E(f_q)
\end{equation*}
where $\omega$ is the K\"ahler form of $N$ and $A$ is area.  The quantity on the left is constant and the first equality occurs if and only if $f_q$ is holomorphic.  If the quantity on the right is also constant and equal to the left at $q_0$, then equalities hold for all $q$. 

\item Let $G(f)\subset V\times N$ be the graph of $f$. To prove that $f$ is holomorphic, Gromov constructs a closed, irreducible algebraic surface in $V\times N$ that contains $G(f)$, thus must equal $G(f)$, thereby proving that $f$ is holomorphic.    We explain one possible  construction.
\begin{enumerate}
\item Assume, for simplicity, that $N$  is compact, hence projective.  Let $Z\subset V\times N\subset \bP^K$ be the union of all algebraic curves $C\subset M\times N$ with the following properties:
\begin{enumerate}
\item $(x_1,y_1),\dots , (x_d,y_d)\in C$, where $x_1,\dots,x_d$ ae the basepoints of the Lefschetz pencil and $y_i = f(x_i)$.
\item There exists $q\in Q$ so that $p_V|_C$  maps $C$ biholomorphically onto $V_q$.
\item The degree of $C$ equals the degree of $G(f_q)$, the graph of $f_q$.
\end{enumerate}
\item
To prove that $Z$ is algebraic, let $\mathcal{X}$ be the Chow variety of curves of degree $d'$ in $\bP^K$, where $d'$ is the degree of $G(f_{q_0})$ in this embedding.  See, for example, Chapter I of \cite{Kollar} for a detailed construction of the Chow variety.   This is an algebraic variety parametrizing algebraic cycles of degree $d'$ in $\bP^K$, and there is a universal cycle $\mathcal{C}\to\mathcal{X}$.  The curve $G(f_{q_0})$ defines a point in $\mathcal{X}$, and the conditions (i) to (iii) above define a closed subvariety $\mathcal{Z}$ of $\mathcal{X}$.  The projection of $\mathcal{C}|_{\mathcal{Z}}$ to $\bP^K$ is a closed subvariety $Z\subset V\times N$.

\item Let $Z_0$ be the irreducible component of $Z$ containing $G(f_0)$, the graph of $f_0$.  If $C\subset Z_0$ is one of the curves near $G(f_0)$, it is the graph of a holomorphic map from some $V_q$ to $N$ in the same homotopy class as $f_q$ and taking $x_1$ to $y_1$.  By uniqueness we must have $C = G(f_q)$, from which it follows that a neighborhood of $G(f_0)$ in $Z_0$ concides with a neighborhood of $G(f_0)$ in $G(f)$.   Since $G(f)\subset Z_0$, it follows that $G(f) = Z_0$.  Sine the latter is algebraic, it follows that $f$ is holomorphic, as desired.
\item If $N$ is not compact, it is still K\"ahler and we can use Barlet's Chow scheme \cite{Barlet}  and Lieberman's compactness theorem  \cite{Lieber} in the same way to prove that $Z_0$ is analytic, hence $f$ is holomorphic.  In the algebraic case, the Hilbert scheme could be used, instead of the Chow variety, to the same effect.
\end{enumerate}
\end{enumerate}
This completes the sketch of proof of Theorem \ref{thm-amplecurve}.   

To prove Corollary \ref{cor-zariskidense}, we first note that it is a consequence of the proof of Theorem \ref{thm-amplecurve} rather than its statement.   Zariski density of the image of the fundamental group implies that for each complex structure in the domain ($V$ or one of the $V_q$), the harmonic map in the homotopy class is unique, see \cite{Sunada} or the \lq\lq split deformation property" of \S 4.6 of \cite{Gromov}.   For the proof of the corollary one deforms $g$ to a harmonic, hence pluriharmonic, map $f$, then obseves that by uniqueness $f_{q_0}$ is holomorphic, then proceeds with the same proof as before, at the last step appealing again to Zariski density for uniqueness.

\medskip
\parskip=0pt

\end{document}